\documentclass[12pt]{article}
\usepackage{srcltx}
\usepackage{amssymb,amsmath,amsfonts,amsthm}
\usepackage{cite}
\newtheorem{lem}{Lemma}
\newtheorem*{theorem}{Theorem}
\newtheorem*{corollary}{Corollary}
\newtheorem{pr}{Proposition}

\begin{document}

\begin{center}\textbf{On Thompson's conjecture for alternating group of large degree}\end{center}

\begin{center}I. B. Gorshkov \footnote{The work is supported by the %!
 Russian Science Foundation (project no. 15-11-10025)}\end{center}

\medskip
\baselineskip=1.5\baselineskip
\textsl{Abstract}: For a finite group $G$, let $N(G)$ denote the set of conjugacy
class sizes of $G$. We show that if every finite group $G$ with trivial
center such that $N(G)$ equals to $N(Alt_n)$, where $n>1361$ and
at least one of numbers $n$ or $n-1$ are decomposed into a sum of two primes, then $G\simeq Alt_n$.


\section{Introduction}

Let $G$ be a finite group.
%The centralizer of $h$ in $G$ is denoted by $C_G(h)$. For $g\in G$, let $g^G$ denote
%the conjugacy class of $G$ containing $g$. If $h\in Aut(G)$ then $|h^G|=|G|/|C_G(h)|$. In particular, if $g\in G$ then
%$|g^G|$ is the order of conjugacy class $g^G$.
Put $N(G)=\{|g^G|\mid g\in G\}$. In 1987 Thompson posed the following conjecture concerning $N(G)$.
\medskip

\textbf{Thompson's Conjecture (see \cite{Kour}, Question 12.38)}. {\it If $L$ is a finite
simple non-abelian group, $G$ is a finite group with trivial
center, and $N(G)=N(L)$, then $G\simeq L$.}

\medskip

It is proved that Thompson's conjecture is valid for many finite simple groups of Lie type (see \cite{Ah}).
Denote the alternating group of degree $n$ by $Alt_n$ and the symmetric group of degree $n$ by $Sym_n$. Alavi and
Daneshkhah \cite{AD} proved that the groups $Alt_n$, where at least one of numbers $n$, $n-1$ or $n-2$ is prime,
are characterized by $N(G)$. This conjecture has also been confirmed for $Alt_{10}, Alt_{16}$,
$Alt_{22}$ and $Alt_{26}$ in \cite{VasT}, \cite{GorTom}, \cite{Xu} and \cite{Liu}, respectively. In \cite{GorA}, the
author showed that if $N(G)=N(Alt_n)$ for $n\geq5$ then $G$ is non-solvable.
In \cite{GorA2}, we obtained some information about composition factors of a group $G$ with $N(G)=N(Alt_n)$, where $n>1361$.

The aim of this paper is to prove the following theorem.

\begin{theorem}
Let $G$ is a finite group with the trivial center such that $N(G)=N(Alt_n)$, $n\geq 27$, $G$ contains a composition factor
$S\simeq Alt_{n-\varepsilon}$, where $\varepsilon$ is a non-negative integer such that the set $\{n-\varepsilon,...,n\}$
does not contain primes and at least one of numbers $n$ or $n-1$ is decomposed into a sum of two primes. Then
$G\simeq Alt_n$.
\end{theorem}

This Theorem and the main result of \cite{GorA2} (see Lemma 13 below) imply immediately the following corollary.

\begin{corollary}
Thompson's conjecture is true for the alternating group of degree $n$ if $n>1361$ and at least one of numbers $n$ or $
n-1$ is decomposed into a sum of two primes.
\end{corollary}

Even positive integers, which can not be decomposed into a sum of two primes, are currently unknown.

\section{Notation and preliminary results}

All notation are standard and can be found in \cite{Gore}.

\begin{lem}[{\rm\cite[Lemma 4]{VasT}}]\label{vas}
Suppose that $G$ is a finite group with the trivial center and $p$ is a prime
from $\pi(G)$ such that $p^2$ does not divide $|x^G|$ for all $x$ in G. Then a Sylow $p$-subgroup
of $G$ is elementary abelian.
\end{lem}

\begin{lem}[{\rm\cite[Teorem 1]{Kantor}}]\label{kant}
Let $G$ be a finite group acting transitively on a set $\Omega$ with $|\Omega|>1$. Then there exist a prime $r$ and
a $r$-element $g\in G$ such that $g$ acts without fixed points on $\Omega$.
\end{lem}

\begin{lem}\label{kant2}
Let $G$ is a finite group, $g\in G$, $|\pi(g)|=1$ and $|g^G|>1$. Then there exists $g'\in G$ such that
$\pi(g)\subseteq\pi(|(g')^G|)$ and $|\pi(g')|=1$.
\end{lem}
\begin{proof}
It follows from Lemma \ref{kant}
\end{proof}

\begin{lem}\label{gfactor}
Let $G=T\leftthreetimes \langle g\rangle$, $(|T|,|g|)=1$, $|\pi(g)|=1$ and $|g^G|>1$. Then there exist
normal subgroups $K$ and $R$ in $G$ such that $K<R$, $R/K$ is a minimal normal subgroup of $G/K$, $|g^{R}|>1$ and
$[g,T]\leq R$.
%$ |g^{(T/K)/R}|=1$.
\end{lem}
\begin{proof}
This follows from Lemma \ref{kant2}, Frattini argument and \cite[Lemma 5.3.2]{Gore}.
\end{proof}
\begin{lem}[{\rm \cite[Theorem 5.2.3]{Gore}}]\label{Gore}
Let $A$ be a $\pi(G)'$-group of automorphisms of an abelian group $G$. Then $G=C_G(A)\times[G,A]$.
\end{lem}

\begin{lem}[{\rm\cite[Lemma 1.6]{GorA2}}]\label{Pord}
 Let $S$ be a non-abelian finite simple group. If $p\in \pi(S)$ then there exist $a\in N(S)$ and $g\in S$ such that
$|a|_p=|S|_p$, $|g^S|=a$ and $|\pi(g)|=1$.
\end{lem}
\begin{lem}[{\rm \cite[Lemma 14]{Vac}}]\label{vac2}
Let $S$ be a non-abelian finite simple group. Any odd element from $\pi(Out(S))$ either belongs to $\pi(P)$ or does not exceed
$m/2$, where $m=max_{p\in\pi(S)}$.
\end{lem}

\begin{lem}[{\rm \cite[Lemma 1.4]{GorA2}}]\label{factorKh}
Let $G$ be a finite group, $K\unlhd G$, $\overline{G}= G/K$, $x\in G$ and $\overline{x}=xK\in G/K$.
The following assertions is validity

(i) $|x^K|$ and $|\overline{x}^{\overline{G}}|$ divide $|x^G|$.

(ii) If $L$ and $M$ are neigboring members of a composition factor of $G$, $L<M$, $S=M/L$, $x\in M$  and
$\widetilde{x}=xL$ is a image of $x$, then $|\widetilde{x}^S|$ divides $|x^G|$.

(iii) If $y\in G, xy=yx$ and $(|x|,|y|)=1$, then $C_G(xy)=C_G(x)\cap C_G(y)$.

(iv) If $(|x|, |K|) = 1$, then $C_{\overline{G}}(\overline{x}) = C_G(x)K/K$.

\end{lem}

\begin{lem}[{\rm \cite[Lemma 2.14]{GorA2}}]\label{ConClass}
Let $n>26$,  $t$ is a prime, $n/2<t\leq n$, $\alpha\in N(Alt_n)$ and $t$ does not divide $\alpha$. Then
$\alpha=|Alt_n|/t|C|$ or $\alpha=|Alt_n|/|Alt_{t+i}||B|$, where $C=C_{Alt_{n-t}}(g)$ for any $g\in Alt_{n-t}$,
$t+i\leq n$ and $B=C_{Alt_{n-t-i}}(h)$ for any $h\in Alt_{n-t-i}$.
\end{lem}

\begin{lem}[{\rm \cite{Hanson}}]\label{Hanson}
The product of $k$ consecutive integers $n(n-1)...(n+k-1)$, which are greater than $k$, contains a prime divisor greater
than $3k/2$ with the exceptions $3\cdot4, 8\cdot9$ and $6\cdot7\cdot8\cdot9\cdot10$.
\end{lem}

\begin{lem}{\rm \cite[Lemma 3.6]{VasBig}}\label{vas2}
Let $s$ and $p$ be distinct primes, $H$ be a semidirect
product of a normal $p$-subgroup $T$ and a cyclic subgroup $\langle g\rangle$ of order $s$,
and $[T, g]\neq 1$. Suppose that $H$ acts faithfully on a vector space $V$ over a field of
positive characteristic $t$ not equal to $p$. If the minimal polynomial of $g$ on $V$
does not equal $x^s-1$, then

$($i$)$ $C_T(g)\neq 1$;

$($ii$)$ $T$ is non-Abelian;

$($iii$)$ $p=2$ and $s = 2^{2^{\delta}}+1$ is a Fermat prime.
\end{lem}
\begin{lem}\label{AnSn}
$N(Sym_n)\neq N(Alt_n)$
\end{lem}
\begin{proof}
The proof is obvious.
\end{proof}

\begin{lem}[{\rm \cite{GorA2}}]\label{GorT}
Let $G$ be a finite group such that $N(G)=N(Alt_n)$, where $n>1361$. Then $G$ has a composition factor isomorphic to an
alternating group $Alt_m$ with $m\leq n$ and the half-interval $(m, n]$ contains no primes.
\end{lem}

\section{Proof of Theorem}

Let $26<n\in \mathbb{N}$, $n,n-1$ and $n-2$ are not prime, at least one of numbers $n$ or $n-1$ are decomposed into a
sum of two primes, $G$ is a finite group with $Z(G)=1$, $N(G)=N(Alt_n)$ and $G$ has a composition factor
$S\simeq Alt_{n-\varepsilon}$ where  is a non-negative integer such that the set $\{n-\varepsilon,...,n\}$
does not contain primes. Let $K$ be a maximal normal subgroup of $G$ such that $G/K$ contains a subgroup which is
isomorphic to $S$. Set $\Omega=\{t\mid t$ is prime, $n/2<t\leq n\}$. It is clear that $S\leq G/K\leq Aut(S)$.

\begin{lem}\label{Order}
$|Alt_n|$ divides $|G|$.
\end{lem}
\begin{proof}
If $n$ is odd then there exist $\alpha,\gamma\in N(G)$ such that $\alpha=n!/2n$, $\gamma=n!/4(n-1)$. If $n$ is even then
there exist $\alpha,\gamma\in N(G)$ such that $\alpha=n!/4n$, $\gamma=n!/2(n-1)$. The statement follows from Lemma
\ref{factorKh}.
\end{proof}

\begin{lem}\label{Omega}
$\Omega\cap\pi(K)=\varnothing$.
\end{lem}
\begin{proof}
Assume that $\Omega\cap\pi(K)\neq\varnothing$. Let $T$ be a normal subgroup of $G$ lying in $K$ such that
$\pi(K/T)\cap\Omega\neq\varnothing$ and $\pi((K/T)/R)\cap\Omega=\varnothing$ for a minimal normal subgroup $R$ of $K/T$.
Let $\widetilde{G}=G/T,~\widetilde{K}=K/T,~t\in\pi(\widetilde{K})\cap\Omega$ and $\widetilde{S}\leq \widetilde{G}$ be a
pre-image of $S$ of minimal order. Since $R$ is minimal normal subgroup of $\widetilde{G}$, we obtain $R=R_1\times...\times R_k$,
where $R_1,...,R_k$ are isomorphic simple groups. Assume that $R$ is non-solvable and $k>1$. By Lemma \ref{Pord} it follows
that in $R_1$ and $R_2$ there exist $y_1$ and $y_2$ such that $|\pi(|y_1|)|=|\pi(|y_2|)|=1$, $t\in\pi(|y_1^{R_1}|)$ and
$t\in \pi(|y_2^{R_2}|)$. We have that $t^2$ divides $|(y_1y_2)^{R_1R_2}|$ and there exists $g\in G$ such that $t^2$
divides $|g^G|$; a contradiction. Hence $k=1$. We have $\widetilde{S}<N_{\widetilde{G}}(R)$. Since $Out(A)$ is solvable
for any finite simple group and $(N_{\widetilde{G}}(R)/C_{\widetilde{G}}(R))/R<Aut(R)$, we obtain
$\widetilde{S}<C_{\widetilde{G}}(R)$. From Lemma \ref{Pord} it follows that there exists $g\in R$ such that
$t\in\pi(|g^R|)$. Using Lemma \ref{factorKh} we obtain that there exists an element $s\in\widetilde{S}$ such that
$|s|\in\Omega\setminus\pi(g)$ and $t\in\pi(|s^{\widetilde{S}}|)$. Therefore, $t^2$ divides $|(sg)^{\widetilde{G}}|$; a
contradiction.

It follows that $R$ is an abelian $t$-group for any prime $t$. Let $g\in \widetilde{S}$ and $|g|\in\Omega\setminus\{t\}$.
It is clear that $g$ centralizes $R$. Therefore, $\widetilde{S}$ centralizes $R$. Since $t\nmid\widetilde{K}/R$, we see
that $C_K(R)=R\times L$. Since $C_K(R)\vartriangleleft \widetilde{G}$, we see that $L\vartriangleleft \widetilde{G}$.
Hence, $L=1$. Therefore, $\widetilde{S}\simeq S$ and
$\widetilde{G}\simeq (R\leftthreetimes\overline{K})\times \widetilde{S}$, where $\overline{K}=\widetilde{K}/R$. If
$\overline{K}\not=1$, then $\overline{K}$ acts faithfully on $R$. By Lemma \ref{kant2} it follows that there exists
$h\in \widetilde{K}$ such that $|\pi(|h|)|=1$ and $t$ divides $|h^{\widetilde{K}}|$. Therefore,
$t^2$ divides $|(hg)^{\widetilde{G}}|$; a contradiction. Thus, $\overline{K}=1$. We have $G\simeq T.(R\times S)$.
Similarly we can obtain that $G\simeq W.(F\times S)$, where $\pi(W)\cap\Omega=\varnothing$ and $\pi(F)\subseteq\Omega$.
It is easy to prove that $F$ is an abelian group.

 Since $Z(G)=1$, we see that $|h^G|>1$ for any $h\in F$. By Lemma \ref{gfactor} if follows that there exist subgroups $L$
and $H$ such that $L\lhd G, L<W$, $H$ is an unique normal subgroup in $\widehat{G}=G/L$, $|x^H|>1$ for some $x\in F$ and
$|h^{(\widehat{G})/H}|=1$ for every $h\in F$. In particular, $(W.S)/L.H<C_{\widehat{G}/H}(F)$ and
$|h^{\widehat{G}}|=|h^{H}|$ for every $h\in F$.

Let $x\in F,~ |x|=t$,~ $r\in \Omega$ and $t\neq r\neq 2^j+1$. Since $H$ is a minimal normal subgroup of $\widehat{G}$, we
have $H\simeq H_1\times...\times H_l$ where $H_1,...,H_l$ are isomorphic finite simple groups. Assume that $H$ is an
elementary abelian $k$-group. It is clear that $C_H(x)\unlhd \widehat{G}$. By the definition of $H$ we obtain that
$C_H(x)=1$. Therefore, $x$ acts on $H$ fixed points freely. In $S$ there exists a Frobenius group
$A=\langle y\rangle\leftthreetimes\langle z\rangle$, where $|y|=r, |z|\in\pi(r-1)\setminus\{k\}$, since $r-1$ is not a
power of prime there exists $|z|$. Hence, $t\in\pi(|y^S|)\cap\pi(|z^S|)$. Let $\widehat{A}<\widehat{G}$ be a pre-image of
$A$ of minimal order. Since $r\nmid|\widehat{W}|$ and by Frattini argument, we obtain
$\langle \widetilde{y}\rangle\lhd \widehat{A}$, where $\widehat{y}\in \widehat{A}$ is a pre-image of $y$ of minimal order.
By Lemma \ref{vas2} it follows that one of groups $C_R(\widehat{y})$ or $C_R(\widehat{z})$ is not trivial.
Let $v\in\{\widehat{y}, \widehat{z}\}$ such that $C_H(v)>1$ and $f\in C_H(v)$. Hence the subgroup $C_{\widehat{G}}(vf)$
intersects with $x^{\widehat{G}}$ and with any Sylow $t$-subgroup of $\widehat{S}$ trivially. Thus
$t^2||(vf)^{\overline{G}}|$; a contradiction.

We have $H$ is non-solvable and $H_1$ is a non-abelian simple group. Since $t>s$ for any $s\in\pi(H)$ and by
Lemma \ref{vac2}, we see that $H_i^x\neq H_i$ for $1\leq i\leq l$. Let $b\in H, b\neq 1$, $b=b_1b_2...b_v$ where if
$b_i,b_j\in H_c$ for any $1\leq c\leq l$, then $i=j$, we shall that $leng(b)=v$. We have $t$ divides $leng(b)$ for any
$b\in C_H(x)$. If $a_1\in H_1$, $a=a_1\cdot a_1^y...a_t^{y^r}$, then $leng(a)=1$ or $leng(a)=r$. Hence,
$a\in C_{\widehat{G}}(\widehat{y})$ and $C_{\widehat{G}}(a)\cap x^{\widehat{G}}=\varnothing$. It follows that
$t\nmid|C_{\widehat{G}}(a\widehat{y})|$. Thus, $t^2$ divides $|(a\widehat{y})^{\widehat{G}}|$; a contradiction.
\end{proof}

By the condition,  $n$ or $n-1$ is equal to $p+q$, where $p$ and  $q$ are primes. We can assume that $p\geq q$.
Further our proof is divided into three propositions.
\begin{pr}\label{pr1}
If $p=q$ then Theorem is true.
\end{pr}
\begin{pr}\label{pr2}
If $q=3$ then Theorem is true.
\end{pr}
\begin{pr}\label{pr3}
$G\simeq Alt_n$.
\end{pr}

\section{Proof of Proposition \ref{pr1}}
Assume that $n$ or $n-1$ is equal to $2p$.
\begin{lem}
$\delta:=n!/2p^2\in N(G)$
\end{lem}
\begin{proof}
Take in $Alt_n$ the element $g=(1,2,...,p)(p+1,...,2p)$. It is obvious that $|C_{Alt_n}(g)|=p^2$.
\end{proof}
\begin{lem}\label{ab}
If $p\nmid\alpha$ for any $\alpha\in N(G)$, then $\alpha=\delta$ or $\alpha=n!/p(n-p)!$.
\end{lem}
\begin{proof}
The proof is obvious.
\end{proof}
\begin{lem}\label{pK}
$p\nmid |K|$
\end{lem}
\begin{proof}
Assume that $p$ divides $|K|$. Let $P\in Syl_p(K)$. By Frattini argument, $N_G(P)/N_K(P)\simeq G/K$. Let
$N=N_G(P)/O_{p'}(N_K(P))$, $\overline{S}$ be a minimal pre-mage in $N$ of $S$, $M$ be a minimal normal subgroup in $N$.
Take an element $g\in N_G(P)$ such that $|g|=p$ and $\overline{g}=gO_{p'}(K)/O_{p'}(K)\in M$. It is obvious that $M<Z(Q)$
for any  $Q\in Syl_p(N)$. In particular, every $p$-element of $N$ centralizes $M$. Since $\overline{S}$ is generated by
all its $p$-elements, we see that $\overline{S}<C_N(M)$. In particular, $t\in\pi(C_N(\overline{g})$ for any $t\in \Omega$.
Using Lemma \ref{factorKh} we get that $p\nmid|g^G|$. Therefore, $|g^G|=\delta$. Combining Lemmas \ref{Omega} and
\ref{factorKh} we obtain that $\pi(|g^G|)\cap\Omega=\varnothing$; a contradiction.
\end{proof}

\begin{lem}\label{gan}
$G\simeq Alt_n$
\end{lem}
\begin{proof}
By Lemmas \ref{Order} and \ref{pK}  it follows that $p^2$ divides $|S|$ and, consequently, $\varepsilon\leq1$.
Assume that $K\not=1$. Let $g\in G, |g|=p, |g^G|\delta, \overline{g}=gK/K$. Then $|\overline{g}^S|=\delta$ or
$|\overline{g}^S|=\delta/n$. Suppose that $|\overline{g}^S|=\delta/n$. Then $g$ acts on $K$ non-trivially. Using
\ref{factorKh} we get $|g^K|=n$. It is easy to prove that for every $t\in\pi(n)$ we have $p$ divides $|n|_t-1$. But it is impossible
since $n-1=2p$; a contradiction. Therefore, $|\overline{g}^S|=\delta$, $\varepsilon=0$ and $K<C_G(g)$. Since $Z(G)=1$
there exists an element $h\in K$  with $|h^G|>1$. Hence, $\delta|$ divides $|(hg)^G|$ and $\delta\neq|(hg^G)|$, however
$\delta$ is maximal by divisibility in $N(G)$; a contradiction. Thus, $K=1$ and $S\simeq Alt_n$. The statement of
Lemma follows from Lemma \ref{AnSn}.
\end{proof}

The proposition is proved.
\section{Proof of Proposition \ref{pr2}}
Assume that $q=3$. Let $g\in G,~ |g|=p$ and $\overline{g}\in S$ be the image of $g$ in $S$. By Lemma \ref{factorKh} we have
$|\overline{g}^S|$ divides $|g^G|$ divides $n!/p$. Suppose that $\varepsilon>0$. Then $|g^K|$ divides $\prod_{n-\varepsilon+1}^n i$. It is
easy to prove that $p$ divides $|\prod_{n-\varepsilon+1}^n i|_t-1$ for any $t\in \pi(\prod_{n-\varepsilon+1}^n i)$; a
contradiction. Therefore, $\varepsilon=0$. The statement of Lemma gets similarly as in Lemma \ref{gan}.

\section{Proof of proposition \ref{pr3}}
Let $p>q>3, g\in G, |g|=p$ and $\overline{g}\in S$ the image of $g$ in $S$.

\begin{lem}\label{alpha}
$\alpha=|g^G|=n!/p(n-p)!$.
\end{lem}
\begin{proof}
We have $\Omega\setminus\{p\}\subset\pi(\overline{g}^S)$. Using Lemma \ref{factorKh} we get
$\Omega\setminus\{p\}\subset \pi(\alpha)$. By Lemma \ref{ConClass} it follows that $\alpha=n!/p|C|$ where $C=C_{Alt_{n-p}}(h)$ for an any element $h\in Alt_{n-p}$.

Assume that $|C|<(n-p)!$. Let $h\in G$ such that $|h^G|=n!/pq$. Since $|G|_p=p$  we can assume that $h\in C_G(g)$.
If $p\in\pi(h)$, then $\alpha$ divides $|h^G|$ and, consequently, $q$ divides $|C|$. Therefore, $\alpha=n!/pq$.
Let $x\in C_G(g)$ such that $|x^G|=n!/pa, 1<a\neq q$. Since $\alpha\nmid|x^G|$ and by Lemma \ref{factorKh}, we obtain that
$p\not\in\pi(|x|)$. Hence, $|x^G|$ and $\alpha$ divide $|(xg)^G|$, but $\alpha$ is maximal by divisibility in $N(G)$; a
contradiction. Therefore, $p\not\in\pi(h)$ and we can assume that $\alpha\neq |h^G|$. Since $|h^G|$ is maximal by
divisibility in $N(G)$ we obtain that $|(gh)^G|=|h^G|$. Therefore, $\alpha$ divides $|(gh)^G|$. Thus, $\alpha=n!/p(n-p)!$.
\end{proof}

Let $C=C_G(g)$, $\alpha=|g^G|$, $\gamma=n!/pq$, $\Theta=\{n!/pb, b|(n-p)!\}\cap N(G)$ and $Q\in Syl_q(C)$.

\begin{lem}\label{vse}
The following statements hold
\begin{enumerate}
\item{$Q\not\leq Z(C)$;}
\item{If $v\in Q\setminus Z(C)$ then $|(gv)^G|=\gamma$;}
\item{$|Q|/|Q\cap Z(C)|=q$;}
\item{If $v\in Q\setminus Z(C)$ and $f$ is a $\{p,q\}'$-element of $C_G(gv)\setminus Z(C)$, then $|(gf)^G|=\gamma$.}
\end{enumerate}
\end{lem}
\begin{proof}
1. By Proposition \ref{pr2} it follows that $n-p>3$. Therefore, there exists $\beta\in\Theta$ such that $\beta_q>\alpha_q$.
Let $h\in C$ and $|h^G|=\beta$. Since $C_G(gh)\leq C, |C_G(gh)|_q<|C|_q$ and by Lemma \ref{factorKh}, we obtain that there
exists $r\in Q\setminus Z(C)$.

2. We have $|(gr)^G|\neq \alpha$. Assume that for any $l\in G$ such that $|(lg)^G|=\gamma$, where $l\in C_G(g)$ and
$\pi(l)=\{t\}$, we have $t\neq q$. Let $b\in C$ and $|b^G|=\gamma$. Since $b$ centralizes some Sylow $q$-subgroup of $C$,
we can assume that $r\in C_G(b)$. Therefore, $\gamma$ divides $|(grb)^G|$ and $|(gr)^G||||(grb)^G|$; a contradiction. Thus,
there exists $r'\in Q$ such that $|(gr')^G|=\gamma$.

Suppose that there exists $r''\in Q\setminus Z(C)$ such that $|(r''g)^G|=\delta\neq \gamma$. Then there exists
$\{p,q\}'$-element $a\in C_G(gr'')$. It is obvious that $|(ga)^G|\neq \gamma$. Therefore, $r'\not \in C_G(ga)$ and
$|(gar'')^G|_q>|(gr'')^G|_q$; a contradiction.

3. Let $v\in C$ and $|v^G|=\beta$. Since $|(gb)^G|=\gamma$ for any $b\in Q\setminus Z(C)$, we see that
$C_G(gv)\cap Q = Z(C)\cap Q$, and, consequently, $|Q|/|Z(C)\cap Q|=q$.

\end{proof}

\begin{lem}\label{epsi5}
$n-p-\varepsilon<5$
\end{lem}
\begin{proof} Assume that $n-p-\varepsilon\geq5$.
From Frattini argument and by Lemma \ref{vse} it follows that there exists a subgroup $H$ in $N_C(Q)$ such that
$H/K\cap H\simeq Alt_{5}$. Let $\overline{N}=N_C(Q)/Z(Q)\cap Z(C) and \overline{Q}=Q/Z(Q)\cap Z(C)$. Note that by
Lemma \ref{vse} we have $|\overline{Q}|=q$. Therefore, $\overline{N}/C_{\overline{N}}(\overline{Q})$ is isomorphic to a
cyclic subgroup of order dividing $q-1$ in $Aut(\overline{Q})$. Hence, $C_{\overline{N}}(\overline{Q})\geq H$. Let
$h_1,h_2\in H\setminus Z(C), |h_1|=2$ and $|h_2|=3$. Since $|(gh_1)^G|=|(gh_2)^G|=\gamma$, we see that
$C_G(gh_1)=C_G(gh_2)$; a contradiction.
\end{proof}

\begin{lem}\label{epsi4}
If $n-p-\varepsilon=4$, then $3$ divides $q-1$.
\end{lem}
\begin{proof} Assume that $n-p-\varepsilon=4$.
By Frattini argument and Lemma \ref{vse} it follows that there exists a subgroup $H<N_C(Q)$ such that
$H/K\cap H\simeq Alt_{4}$. Let $\overline{N}=N_C(Q)/Z(Q)\cap Z(C)$ and $\overline{Q}=Q/Z(Q)\cap Z(C)$. Note that by Lemma
\ref{vse} we have $|\overline{Q}|=q$. Therefore, $\overline{N}/C_{\overline{N}}(\overline{Q})$ is isomorphic to a cyclic
subgroup of order dividing $q-1$ in $Aut(\overline{Q})$. If $C_{\overline{N}}(\overline{Q})\geq H$, then similarly as in
Lemma \ref{epsi5} we obtain a contradiction. Thus, $|H/C_{\overline{N}}(\overline{Q})\cap H|=3$.
\end{proof}

\begin{lem}\label{Pi}
The number $\Pi=\prod_{n-\varepsilon+1}^{n} i$ is not divisible by primes which are larger than $q$.
\end{lem}
\begin{proof}
Suppose that there exists a prime $t$ dividing $\Pi$ and $t>q$. Then $t$ divides only one number from the set
$\{n-\varepsilon+1,...,n\}$. In particular, $|\Pi|_t-1<p$. By Lemma \ref{alpha} it follows that $|g^G|=\alpha$ and
$|\overline{g}^{G/K}|=(n-\varepsilon)!/(n-\varepsilon-p)!$. Since $K$ is not divisible by $p$, we obtain
$|g^K|=\Pi/(\prod_{n-p-\varepsilon+1}^{n-p}i)$. Therefore, $t$ divides $|g^K|$. Let $H$ be a maximal normal subgroup in $K$
such that $t$ divides $|\widetilde{g}^{\widetilde{K}}|$, where $\widetilde{K}=K/H$ and $\widetilde{g}=gH\in G/H$. Let
$N$ be a minimal normal subgroup of $\widetilde{K}$. By maximality of $H$, $t$ divides $|\widetilde{g}^N|$. We have
$N\simeq N_1\times...\times N_k$, where $N_1,..., N_k$ are isomorphic simple groups. It is easy to prove that
$|\widetilde{g}^N|_t\geq p+1$; a contradiction.
\end{proof}
\begin{lem}\label{epsilon}
If  $\varepsilon>0$ then $\varepsilon\leq4, q\leq7$ and if $n-\varepsilon-p=2$ then $n-p\leq5$.
\end{lem}
\begin{proof}
The statement of Lemma follows from Lemmas \ref{Hanson}, \ref{epsi5}, \ref{epsi4} and  \ref{Pi}.
\end{proof}
\begin{lem}
$\varepsilon=0$
\end{lem}
\begin{proof}
Assume that $\varepsilon>0$. Then $|g^K|=\alpha/ |\overline{g}^{G/K}|$. It is clear that if $\varepsilon>1$, therefore
there exists an odd prime divisor $t$ of $|g^K|$. We can show that $p$ divides $|g^K|_t-1$. If $t$ divides only one
number from the set $\Lambda=\{n-\varepsilon+1,...,n\}$, then we can obtain a contradiction as in Lemma \ref{Pi}. By
Lemma \ref{epsilon} it follows that $\varepsilon\leq4$. Since $t$ divides two numbers from the set $\Lambda$ we have $t=3$.
Therefore, $|g^K|_t\leq t^{l+1}$, where $t^l$ divides only one number from the set $\Lambda$. Analyzing the residue of
the division of $t^l$ on $p$ it is easy to get a contradiction.
\end{proof}

\begin{lem}
$G\simeq Alt_n$.
\end{lem}
\begin{proof}
Since $|\overline{g^S}|=|g^S|$, we see that $K<C_G(g)$. Therefore, $K.S$ is a central extension. Since $Z(G)=1$, it
follows that $Z(K)=1$. Assume that $K\not=1$. Let $h\in K$ be an element of prime order $t$. Then $t$ is co prime to
the number $\theta\in \{n-1,n-2\}$. There exists $x\in G$ such that $|x|=\theta$ or $|x|=\theta/2 $ and $|x^S|=n!/\theta$.
It's obvious that $|x^S|$ is maximal by divisibility in $N(G)$. By Lemma \ref{factorKh} we obtain $|(hx')^G|>|x^S|$,
where $x'G$ is a pre-image in $G$ of $x$; a contradiction. Thus $K=1$. By Lemma \ref{AnSn} we obtain the statement
of Lemma.
\end{proof}
The Proposition and Theorem proved.

Corollary follows from Lemma \ref{GorT} and Theorem.

\thispagestyle{empty}
I.B. Gorshkov\\
N. N. Krasovskii Institute of Mathematics and Mechanics\\
16 S. Kovalevskaya St., Ekaterinburg, 620990, Russia\\
E-mail address: ilygor8@gmail.com
\end{document}